\documentclass[reqno,11pt]{amsart}
\usepackage{amsmath,amssymb,amsthm,mathrsfs}
\usepackage{epsfig,color}
\usepackage[latin1]{inputenc}
\usepackage[english]{babel}
\usepackage[numbers]{natbib}
\usepackage[colorlinks, citecolor=blue, linkcolor=red]{hyperref}

\numberwithin{equation}{section}

\newtheorem{ackn}{Acknowledgments\!}

\def\k{\kappa}

\def\00{{\bf 0}}

\def\RR{\mathbb R}

\renewcommand{\d}{\mathrm{d}}

\newcommand{\D}{\Delta}

\renewcommand{\L}{\Lambda}

\newcommand{\eps}{{\varepsilon}}

\def\p{\mathbf{p}}

\def\D{\mathrm{D}}
\def\k{\mathrm{\bf k}}
\def\p{\mathrm{\bf p}}
\def\L{\mathcal{L}}

\newtheorem*{theorem*}{Theorem}
\newtheorem{theorem}{Theorem}[section]
\newtheorem{lemma}[theorem]{Lemma}
\newtheorem{proposition}[theorem]{Proposition}

\newtheorem{corollary}[theorem]{Corollary}
\newtheorem{remark}[theorem]{Remark}

\begin{document}
    \title[Rigidity of solutions to singular/degenerate semilinear critical equations]{Rigidity of solutions to \\singular/degenerate semilinear critical equations}

  \date{}

\author{Giovanni Catino, Dario D. Monticelli, Alberto Roncoroni}

\address{G. Catino, Dipartimento di Matematica, Politecnico di Milano, Piazza Leonardo da Vinci 32, 20133, Milano, Italy.}
\email{giovanni.catino@polimi.it}

\address{D. Monticelli, Dipartimento di Matematica, Politecnico di Milano, Piazza Leonardo da Vinci 32, 20133, Milano, Italy.}
\email{dario.monticelli@polimi.it}

\address{A. Roncoroni, Dipartimento di Matematica, Politecnico di Milano, Piazza Leonardo da Vinci 32, 20133, Milano, Italy.}
\email{alberto.roncoroni@polimi.it}

\begin{abstract} 
This paper deals with singular/degenerate semilinear critical equations which arise as the Euler-Lagrange equation of Caffarelli-Kohn-Nirenberg inequalities in $\mathbb{R}^d$, with $d\geq 2$. We prove several rigidity results for positive solutions, in particular we classify solutions with possibily infinite energy when the intrinsic dimension $n$ satisfies $2<n<4$.
\end{abstract}

\maketitle

\begin{center}


\medskip

\centerline{\textbf{AMS subject classification:} 35J91, 35B33, 58J05, 35R01.}
\end{center}

\

\section{Introduction}

The study of sharp functional inequalities and their extremal functions is a central theme in analysis, with profound connections to the classification of solutions of nonlinear partial differential equations. In this context, the \emph{Caffarelli-Kohn-Nirenberg inequalities} (CKN) form a two-parameter family of interpolation inequalities that generalize the classical Sobolev one by incorporating weight singularities or degeneracies.

In their simplest form, the CKN inequalities, proved in \cite{CKN}, read as
\begin{equation}\label{CKN}
\left(\int_{\mathbb{R}^d}\frac{|u|^p}{|x|^{bp}}\, dx\right)^{\frac{2}{p}}\leq C_{a,b}\int_{\mathbb{R}^d}\frac{|\nabla u|^2}{|x|^{2a}}\, dx\, , \quad \text{ for any } u\in\mathcal{D}^{a,b}(\mathbb{R}^d)\, ,
\end{equation}
where 
$$
\mathcal{D}^{a,b}(\mathbb{R}^d)=\left\{u\in L^p(\mathbb{R}^d,|x|^{-b}\, dx)\, : \, |x|^{-a}|\nabla u|\in L^2(\mathbb{R}^d,dx)  \right\}\, .
$$
In \eqref{CKN}, the parameters $a, b \in \mathbb{R}$ satisfy the conditions 
\begin{equation}\label{a-b}
a\leq b \leq a+1 \quad \text{ for } d\geq 3 \, , \quad \text{and}\quad  a< b \leq a+1 \quad \text{ for } d=2\,. 
\end{equation}
In addition we assume
\begin{equation}\label{a_c}
a<a_c := \frac{d-2}{2}\,.
\end{equation}
Indeed, it is known that for $a = a_c$ the inequality \eqref{CKN} fails to be true (see \cite{CKN}), while the dual case $a>a_c$ is related to the case $a<a_c$ via the ``modified inversion symmetry" that can be found in \cite[Theorem 1.4]{CW} (see also \cite{DEL}). Moreover the constant $C_{a,b}$ appearing in \eqref{CKN} is the optimal one and it depends only on $a,b$ and $d$. The exponent
\begin{equation}\label{p_v1}
p:=\frac{2d}{d-2+2(b-a)}
\end{equation} 
is chosen so that the inequality \eqref{CKN} is scale-invariant.

Inequality \eqref{CKN} interpolates, for $d\geq 3$, between the Sobolev inequality ($a=b=0$) and the weighted Hardy inequalities corresponding to $b=a+1$ (see e.g.  \cite{CW, DEL} for more details). We mention that the equality cases in \eqref{a-b}. i.e. $a=b$ (i.e $p=2^\ast$) and $b=a+1$ (i.e $p=2$). If $a=b<0$, $d\geq 3$ or $b=a+1$, $d\geq 2$ then equality \eqref{CKN} is never achieved in $\mathcal{D}^{a,b}(\mathbb{R}^d)$, while it is realized if $a=b\geq0$, $d\geq 3$ (see \cite{CW} for further details).

When $a = b = 0$,  \eqref{CKN} reduces to the Sobolev inequality, whose extremal functions were classified independently in \cite{aubin}, \cite{talenti} and \cite{rodemich}. 
In particular, for $a = b = 0$, they showed that the extremal functions are radial and given by the so called {\em bubbles}. The Euler-Lagrange equation associated with the Sobolev inequality is the critical semilinear (Yamabe) equation
\begin{equation}\label{eqcritclass}
\Delta u + u |u|^{\frac{4}{d - 2}} = 0 \quad \text{in } \mathbb{R}^d\,.
\end{equation}
Classical results in \cite{GNN} and \cite{obata} showed that all {\em positive} $C^2$-solutions of \eqref{eqcritclass} are radially symmetric under a suitable decay assumption at infinity. Subsequent works removed such \emph{a priori} assumption. In particular, in \cite{CGS} the authors established a landmark Liouville-type theorem for \eqref{eqcritclass} that deduced radial symmetry and classification of positive solutions without any decay (or energy) assumption. Their approach combined a Kelvin transform with the method of moving planes, using a delicate asymptotic analysis to handle the singular behavior at infinity. This was later refined in \cite{LZ} and in \cite{CL}, where the authors obtained the rigidity via maximum principle techniques. Thus, for the classical critical equation \eqref{eqcritclass}, the classification of positive solutions was eventually achieved in full generality, showing that the only positive solutions are the bubbles up to translations and scalings.

When one introduces non-zero parameters $a,b$ in \eqref{CKN}, the nature of extremal functions and solutions to the corresponding Euler-Lagrange equations becomes markedly richer. The associated Euler-Lagrange equation is the following weighted semilinear elliptic equation
\begin{equation}\label{CKN-eq}
\mathrm{div}\left(|x|^{-2a}\nabla u\right) + |x|^{-bp} |u|^{p-2}u = 0\quad \text{in } \mathbb{R}^d, 
\end{equation}
where $p$ is given by \eqref{p_v1}.
It is well known that \eqref{CKN-eq}  admits a family of explicit radial solutions  generalizing the Sobolev bubbles (see e.g. \cite[Section 2.3]{CW}). For example, under suitable relations between $a, b, d$, any positive radial solution of \eqref{CKN-eq}  must take the form:
\begin{equation}\label{Tal_CKN}
\mathcal{U}(x):=\left(\frac{d(p-2)(a_c-a)^2}{1+a-b}\right)^{\frac{1}{p-2}}\left(1+|x|^{(p-2)(a_c-a)}\right)^{-\frac{2}{p-2}}\, ,
\end{equation}
up to scaling.  A fundamental question is whether these radial functions characterize all possible positive solutions of \eqref{CKN-eq}, or whether non-radial solutions can exist. This question is intimately connected to the phenomenon of {\em symmetry breaking} for extremal functions of \eqref{CKN}. It was discovered that, unlike the pure Sobolev case, extremal functions for  \eqref{CKN} need not always be radially symmetric. Indeed, in \cite{FS} the authors showed that for certain ranges of the parameters $a,b\in\mathbb{R}$, the extremal functions of  \eqref{CKN} are not radially symmetric. More precisely, defining an auxiliary parameter $\alpha$ and an {\em intrinsic dimension} $n$ by
\begin{equation}\label{alpha-n}
\alpha:=\frac{(1+a-b)(a_c-a)}{a_c-a+b},\quad \text{ and } \quad  n:=\dfrac{d}{1+a-b}\, ,
\end{equation}
they proved that if $d\geq 3$, $p>2$, $a,b\in\mathbb{R}$ satisfy \eqref{a-b}-\eqref{a_c} and if 
$$
\alpha> \sqrt{\dfrac{d-1}{n-1}}\, , 
$$
then extremal functions are not radially symmetric. This result provided concrete examples of symmetry breaking in weighted critical equations, a striking difference from the unweighted case. On the other hand, it left open the complementary question: in the subcritical regime (i.e. when $\alpha$ is below the threshold), are all extremal functions necessarily radial? In \cite{DEL} the authors introduced a novel nonlinear flow on a cylinder (related to fast diffusion) to prove an optimal Liouville-type theorem for positive solution of \eqref{CKN-eq}. Their result showed that for $p\in(2,2^\ast)$, $a,b\in\mathbb{R}$ satisfying \eqref{a-b}-\eqref{a_c} and
\begin{equation}\label{FS_alpha}
\alpha\leq \sqrt{\dfrac{d-1}{n-1}}\,,
\end{equation}
every positive solution of \eqref{CKN-eq} that belongs to the energy space $\mathcal D^{a,b}(\mathbb{R}^d)$ is indeed radially symmetric and given by \eqref{Tal_CKN} up to scaling. This settled the question of symmetry for extremals of \eqref{CKN} and provided a complete classification of finite energy solutions in the symmetric regime.

Despite this progress, a fundamental open problem remained: can one classify all positive solutions of \eqref{CKN-eq}  without the finite energy condition (or with a weak decay condition at infinity)? In the critical Sobolev case ($a=b=0$), as mentioned above, this was achieved in \cite{CGS,CL,LZ}. The case $a\geq 0$   has been considered in \cite{CC} by using the moving planes method as in \cite{CL}. In the case $a=b=0$, a proof based on integral estimates that can be applied also in the Riemannian setting when the Ricci curvature is non-negative, has been  recently obtained in \cite{CaMo} if $d=3$, and was extended to $d=4,5$ in \cite{CFP}.

We also mention that an analogue picture holds for the quasilinear version of  \eqref{eqcritclass}, i.e. 
$$
\Delta_p u + u |u|^{\frac{d(p-2)+2p}{d-p}} = 0 \quad \text{in } \mathbb{R}^d\,,
$$
where $1<p<d$ and $\Delta_p$ is the usual $p$-Laplace operator. In \cite{DMMS,Sci,Vet} the method of moving planes has been used to classify finite energy (positive) solutions, while in \cite{CMR,Ou,Vet_bis} the integral approach has been adapted to classify all positive solutions provided some restrictions on the dimension (see also \cite{CMR,cinesi} for the Riemannian extensions). Finally, also the subelliptic critical equation in the Heisenberg group (see \cite{CLMR,FV,Praj} and \cite{JL} for the finite energy result) and in the Sasakian setting (see \cite{CMWR}) have been considered.

%
%
%
%
%\begin{figure}
%\begin{center}\scalebox{0.7}{\includegraphics{untitled.pdf}}\end{center}
%\caption{$d=3$}
%\end{figure}
%

In this paper, we tackle the classification of positive solutions to \eqref{CKN-eq} by carefully adapting the approach in \cite{DEL} in order to avoid any \emph{a priori} energy or decay conditions. Our main results are complete classifications of all positive solutions to \eqref{CKN-eq} in the symmetry regime, without assuming $u \in \mathcal D^{a,b}(\mathbb{R}^d)$ in low intrinsic dimension or under suitable decay assumptions at infinity. In particular our first result is the following

\begin{theorem}\label{teo1} Let $u\in\mathcal{D}^{a,b}_{\mathrm{loc}}(\mathbb{R}^d)$ be a positive solution of \eqref{CKN-eq}, with $p\in (2,2^*)$ and  $a,b$ satisfying \eqref{a-b}-\eqref{a_c}. If
$$
2<n<4\quad\text{and}\quad \alpha\leq \sqrt{\dfrac{d-1}{n-1}}\,,
$$
then $u(x)=\mathcal{U}(x)$, where $\mathcal{U}$ is given by \eqref{Tal_CKN}, up to scaling.
\end{theorem}

We remark that very recently in \cite{CirPol} the authors have obtained closely related classification results in $\mathbb{R}^d$ and in convex cones (see \cite{CirPol} and also \cite{CFR,CPP,LPT} for the critical Sobolev case in convex cones) under the following assumptions
$$
\frac{5}{2}<n<5\quad\text{and}\quad \alpha\leq \sqrt{\dfrac{d-2}{n-2}}\,.
$$
To our knowledge, these are the first rigidity theorems for weighted critical equations that hold without assuming decay or energy conditions. In addition, we show that the rigidity result in Theorem \ref{teo1} holds true in higher intrinsic dimension, provided a suitable pointwise decay condition at infinity holds.

\begin{theorem}\label{teo3} Let $u\in\mathcal{D}^{a,b}_{\mathrm{loc}}(\mathbb{R}^d)$ be a positive solution of \eqref{CKN-eq}, with $p\in (2,2^*)$ and  $a,b$ satisfying \eqref{a-b}-\eqref{a_c}. If 
$$
n\geq 4\quad\text{and}\quad\alpha\leq \sqrt{\dfrac{d-1}{n-1}}\,,
$$
and
$$
u(x)\leq C|x|^{\sigma\alpha}\quad\text{for some }\sigma<-\frac{(n-2)(n-6)}{2(n-4)}\,,
$$
outside a compact set, then $u(x)=\mathcal{U}(x)$, where $\mathcal{U}$ is given by \eqref{Tal_CKN}, up to scaling.
\end{theorem}

Combining Theorems \ref{teo1} and \ref{teo3} we obtain the following corollaries.
\begin{corollary}\label{cor1} Let $u\in\mathcal{D}^{a,b}_{\mathrm{loc}}(\mathbb{R}^d)$ be a positive solution of \eqref{CKN-eq}, with $p\in (2,2^*)$ and  $a,b$ satisfying \eqref{a-b}-\eqref{a_c}. If
$$
d\geq 2\quad\text{and}\quad\alpha\leq \sqrt{\dfrac{d-1}{n-1}}\,,
$$
and
$$
u(x)\leq C|x|^{-\frac{d-2-2a}{2}}\,,
$$
outside a compact set, then $u(x)=\mathcal{U}(x)$, where $\mathcal{U}$ is given by \eqref{Tal_CKN}, up to scaling.
\end{corollary}

\begin{corollary}\label{cor2} Let $u\in\mathcal{D}^{a,b}_{\mathrm{loc}}(\mathbb{R}^d)$ be a {\em bounded} positive solution of \eqref{CKN-eq}, with $p\in (2,2^*)$ and  $a,b$ satisfying \eqref{a-b}-\eqref{a_c}. If 
$$
2<n\leq 6 \quad\text{and}\quad\alpha\leq \sqrt{\dfrac{d-1}{n-1}}\,,
$$
then $u(x)=\mathcal{U}(x)$, where $\mathcal{U}$ is given by \eqref{Tal_CKN} up to scaling.
\end{corollary}

We stress the fact that the limiting exponent in Theorem \ref{teo3} is strictly
larger than $-\frac{d-2-2a}{2}$. Moreover, we note that the exponent $-\frac{d-2-2a}{2}$ is
the threshold decay in order for a radial solution to have finite energy. In the finite energy case, we also recover the result in \cite{DEL} with a simpler proof. 
 
\begin{theorem}\label{teo2} Let $u\in\mathcal{D}^{a,b}(\mathbb{R}^d)$ be a positive solution of \eqref{CKN-eq}, with $p\in (2,2^*)$ and  $a,b$ satisfying \eqref{a-b}-\eqref{a_c}. If
$$
d\geq 2\quad\text{and}\quad\alpha\leq \sqrt{\dfrac{d-1}{n-1}}\,,
$$
then $u(x)=\mathcal{U}(x)$, where $\mathcal{U}$ is given by \eqref{Tal_CKN}, up to scaling.
\end{theorem}

The proof of our results is a careful adaptation of the argument in \cite{DEL}. One key observation of this paper, contained in Lemma \ref{div_id}, is to rewrite the fundamental Bochner quantity $\k[\p]$ defined in \eqref{def_k}, introduced and analysed in \cite{DEL}, as a divergence of a vector field \`a la Obata \cite{obata}. This fact leads us to prove the rigidity without any higher order asymptotic estimates on the solution at infinity and at the origin which is assured from the finite energy assumption (see \cite[Appendix]{DEL} and also \cite{CirCor,CCR,SV}), but just exploiting the equation and at most a $C^0$-control on the solution.

%\

%\subsection*{Organization of the paper} In Section \ref{prel}

\

\section{Preliminaries}\label{prel}

\subsection{Notations and a key divergence formula}\label{Not}

For $d\geq 2$ we consider \emph{spherical coordinates:} given $x\in\mathbb{R}^d$ we define 
$$
r=|x| \quad \text{and} \quad \theta=\frac{x}{|x|}\, ,
$$
we consider the \emph{change of variables} $r\mapsto r^{\alpha}$, where $\alpha$ is given by \eqref{alpha-n}, and we define 
\begin{equation}\label{def_w}
u(r,\theta)=:w(r^\alpha,\theta)\, , \quad \text{ for every } (r,\theta)\in(0,+\infty)\times\mathbb{S}^{d-1}\, . 
\end{equation}
In the new variables the derivatives are given by
$$
\mathrm{D}w=\left(\alpha w',\frac{1}{r}\nabla_\theta w\right)\, ,
$$
where $w'$ denotes the derivative with respect to the radial variable $r\in(0,+\infty)$ and $\nabla_\theta$ denotes the gradient with respect to the angular variable $\theta\in\mathbb{S}^{d-1}$. As pointed out in \cite[Section 3]{DEL} the CKN inequalities \eqref{CKN} read as 
$$
\alpha^{1-\frac{2}{p}}\left(\int_{(0,+\infty)\times\mathbb{S}^{d-1}}|w|^p\, d\mu\right)^{\frac{2}{p}} \leq C_{a,b}\int_{(0,+\infty)\times\mathbb{S}^{d-1}} |\D w|^2\, d\mu\, , 
$$
where 
$$
d\mu=r^{n-1}\,dr\,d\theta\, ,
$$
and 
\begin{equation*}
|\mathrm{D}w|^2=\alpha^2|w'|^2+\dfrac{|\nabla_\theta w|^2}{r^2}\,  .
\end{equation*}
Moreover, as shown in \cite[Section 3]{DEL}, if $u$ solves \eqref{CKN-eq} then $w$ given by \eqref{def_w} solves 
\begin{equation}\label{eq_w}
-\mathcal{L}w=w^{p-1} \quad \text{ in } (0,+\infty)\times\mathbb{S}^{d-1}\, ,
\end{equation}
where $p$ is given by \eqref{p_v1} or, equivalently using the definition of $n$ in \eqref{alpha-n}, 
\begin{equation}\label{p_v2}
p=\frac{2n}{n-2}\, , 
\end{equation}
 and $\L$ is given by
\begin{equation}\label{def_L}
\mathcal{L}w:=\D_i\D_i w=\alpha^2 w'' +\alpha^2\frac{n-1}{r}w'+\frac{\Delta_\theta w}{r^2}\, ,
\end{equation}
where we used the Einstein convention over repeated indices and where $\Delta_{\theta}$ denotes the Laplace-Beltrami operator of $\mathbb{S}^{d-1}$.

%We consider the metric measure space
%$$
%((0,\infty)\times\mathbb{S}^{d-1},g,d\mu)\, ,
%$$
%where
%$$
%g=\frac{1}{\alpha^2}dr+r^2g_{\mathbb{S}^{d-1}}\, , \quad 
%$$
%and
%$$
%d\mu=r^{n-1}\,dr d\theta\, ,
%$$
%where $d\theta$ denotes the Lebesgue measure on $\mathbb{S}^{d-1}\subset\mathbb{R}^d$.
%
%Then the Laplace-Beltrami operator is 
%$$
%\Delta_g=\alpha^2 w'' +\alpha^2\frac{d-1}{r}w'+\frac{\Delta w}{r^2}\, ,
%$$ 
%che non corrisponde a $\mathcal{L}$. 

Besides the change of variables the other fundamental ingredient in \cite{DEL} is the following auxiliary function, called \emph{the pressure function}, 
\begin{equation}\label{def_p}
\p:=(n-1) w^{-\frac{2}{n-2}}\, ,
\end{equation}
%Then 
%\begin{equation}\label{def_L}
%\mathcal{L}\p=\alpha^2  \p'' +\alpha^2\frac{n-1}{r}\p'+\frac{\Delta_\theta \p}{r^2}\,,
%\end{equation}
%and 
%\begin{equation}\label{Dp}
%|\mathrm{D}\p|^2=\alpha^2|\p'|^2+\dfrac{|\nabla_\theta\p|^2}{r^2}\,  .
%\end{equation}
together with \emph{the Bochner quantity}
\begin{equation}\label{def_k}
\k[\p]:=\frac{1}{2}\mathcal{L}|\mathrm{D}\p|^2-\langle\D\p,\D\L\p\rangle-\frac{1}{n}\left(\L\p\right)^2\, ,
\end{equation}
where 
$$
\langle\D\p,\D\L\p\rangle=\alpha^2\p'\left(\L\p\right)'+\frac{\nabla_\theta\p\cdot\nabla_\theta\left(\L\p\right)}{r^2}\, ,
$$
and $\cdot$ denotes the standard product in $\mathbb{S}^{d-1}$.

As mentioned in the introduction a key difference with respect to the proof in \cite{DEL} is presented in the following Lemma.

\begin{lemma}\label{div_id}
Let $u$ be a positive solution of \eqref{CKN-eq} and $\p\in C^3\left((0,+\infty)\times\mathbb{S}^{d-1}\right)$ given by \eqref{def_p}, then 
\begin{equation}\label{eq_p}
\L\p=\dfrac{2(n-1)^2}{n-2}\p^{-1}+\dfrac{n}{2}|\D\p|^2\p^{-1} \quad \text{ in } (0,+\infty)\times\mathbb{S}^{d-1}\, ,
\end{equation}
and 
\begin{equation}\label{div_id_formula}
\p^{1-n}\k[\p]=\D_i\left( \frac{1}{2}\p^{1-n}\D_i|\D\p|^2-\frac{1}{n}\p^{1-n}\L\p\D_i\p\right)\quad \text{ in } (0,+\infty)\times\mathbb{S}^{d-1}\, . 
\end{equation}
\end{lemma}

\begin{proof}
We start by proving \eqref{eq_p}. From the definition of $\p$ \eqref{def_p} and the definiton of $\mathcal{L}$ \eqref{def_L} we have 
\begin{align*}
\mathcal{L}\p=&\frac{2n(n-1)\alpha^2}{(n-2)^2}w^{-\frac{2(n-1)}{n-2}}(w')^2-\frac{2(n-1)}{n-2}w^{-\frac{n}{n-2}}\mathcal{L}w + \frac{2n(n-1)}{(n-2)^2}w^{-\frac{2(n-1)}{n-2}}\frac{|\nabla_\theta w|^2}{r^2}\\
=&\frac{2n(n-1)}{(n-2)^2}w^{-\frac{2(n-1)}{n-2}}\left( \alpha^2 (w')^2+\frac{|\nabla_\theta w|^2}{r^2}\right)+ \frac{2(n-1)}{n-2}w^{\frac{2}{n-2}}\, ,
\end{align*}
where we used \eqref{eq_w}. From \eqref{def_p} we immediately see that
$$
\frac{2(n-1)}{n-2}w^{\frac{2}{n-2}}=\frac{2(n-1)^2}{n-2}\p^{-1}\, ,
$$
moreover, 
$$
\alpha^2\frac{2n(n-1)}{(n-2)^2}w^{-\frac{2(n-1)}{n-2}}(w')^2=\alpha^2\frac{n}{2}\p^{-1}|\p'|^2
$$
and 
$$
\frac{2n(n-1)}{(n-2)^2}w^{-\frac{2(n-1)}{n-2}}\frac{|\nabla_\theta w|^2}{r^2}=\frac{n}{2}\p^{-1}\frac{|\nabla_\theta \p|^2}{r^2}\, . 
$$
Summing up, we get 
$$
\mathcal{L}\p=\frac{n}{2}\p^{-1}\left(\alpha^2|\p'|^2+\frac{|\nabla_\theta\p|^2}{r^2}\right) +\frac{2(n-1)^2}{n-2}\p^{-1}\, , 
$$
and \eqref{eq_p} follows immediately. Regarding \eqref{div_id_formula} we compute 
\begin{align}\label{1.0}
\D_i\left( \frac{1}{2}\p^{1-n}\D_i|\D\p|^2-\frac{1}{n}\p^{1-n}\L\p\D_i\p\right)=&\frac{1-n}{2}\p^{-n}\langle\D\p,\D|\D\p|^2\rangle + \frac{1}{2}\p^{1-n}\mathcal{L}|\D\p|^2 \nonumber \\
& -\frac{1-n}{n}\p^{-n}\mathcal{L}\p|\D\p|^2- \frac{1}{n}\p^{1-n}\langle\D\mathcal{L}\p,\D\p\rangle \nonumber \\ 
&-\frac{1}{n}\p^{1-n}\left(\mathcal{L}\p\right)^2\, .
\end{align}
From \eqref{eq_p} we have 
$$
\D|\D\p|^2=\frac{2}{n}\mathcal{L}\p\D\p+\frac{2}{n}\p\D\mathcal{L}\p\, , 
$$
hence 
\begin{equation}\label{1.1}
\frac{1-n}{2}\p^{-n}\langle\D\p,\D|\D\p|^2\rangle =\frac{1-n}{n}\p^{-n}\mathcal{L}\p|\D\p|^2+\frac{1-n}{n}\p^{1-n} \langle\D\p,\D\mathcal{L}\p\rangle\, .
\end{equation}
By using \eqref{1.1} in \eqref{1.0} we obtain 
\begin{align*}
\D_i\left( \frac{1}{2}\p^{1-n}\D_i|\D\p|^2-\frac{1}{n}\p^{1-n}\L\p\D_i\p\right)=& \frac{1}{2}\p^{1-n}\mathcal{L}|\D\p|^2 - \p^{1-n}\langle\D\mathcal{L}\p,\D\p\rangle-\frac{1}{n}\p^{1-n}\left(\mathcal{L}\p\right)^2\, ,
\end{align*}
and \eqref{div_id_formula} follows immediately from the definition of $\k$ in \eqref{def_k}.
\end{proof}

\subsection{Integral estimates} From the differential identity obtained in Lemma \ref{div_id} we deduce several integral estimates which will be used in the proofs of our main results. The first one is substantially contained in \cite[Section 5]{DEL}.

\begin{proposition}\label{prop-del}
Assume $n>d\geq 2$ and let $\p\in C^3\left((0,\infty)\times\mathbb{S}^{d-1}\right)$. Then 
\begin{align}\label{0.0}
\int_{(0,+\infty)\times\mathbb{S}^{d-1}}\p^{1-n}\k[\p]\eta^s\, d\mu\geq& \left(\frac{n-1}{n}\right)\alpha^4\int_{(0,+\infty)\times\mathbb{S}^{d-1}}\p^{1-n}\left| \p''-\frac{\p'}{r} -\frac{\Delta_\theta\p}{\alpha^2(n-1)r^2}\right|^2\eta^s\, d\mu\nonumber \\
& + 2\alpha^2\int_{(0,+\infty)\times\mathbb{S}^{d-1}}\p^{1-n}\frac{1}{r^2}\left| \nabla_\theta \p' -\frac{\nabla_\theta \p }{r}\right|^2\eta^s\, d\mu \nonumber \\
& + (n-2)\left(\dfrac{d-1}{n-1}-\alpha^2\right)\int_{(0,+\infty)\times\mathbb{S}^{d-1}}\p^{1-n} \frac{1}{r^4} |\nabla_\theta\p |^2\eta^s\, d\mu\, . 
\end{align}
for every positive $\eta=\eta(r)\in C^{\infty}_c(0,+\infty)$ and for every $s\geq 0$. In particular, if \eqref{FS_alpha} is in force we have
\begin{equation}\label{geq0}
\int_{(0,+\infty)\times\mathbb{S}^{d-1}}\p^{1-n}\k[\p]\eta^s\, d\mu\geq 0 \, . 
\end{equation}
Moreover, if \eqref{FS_alpha} is in force we have
\begin{equation}\label{0.00}
\int_{(0,+\infty)\times\mathbb{S}^{d-1}}\p^{1-n}\k[\p]\, d\mu= 0 \quad \Longleftrightarrow \quad  u(x)=\mathcal{U}(x)\, ,
\end{equation}
where $u$ is related to $\p$ via \eqref{def_p}-\eqref{def_w}, and $\mathcal{U}$ is given by \eqref{Tal_CKN}, up to scaling.
\end{proposition}

\begin{proof} Let $d\geq 3$. From  \cite[Lemma 5.1]{DEL} we have
\begin{align}\label{eq22}
\k[\p]=\left(\frac{n-1}{n}\right)\alpha^4\left| \p''-\frac{\p'}{r} -\frac{\Delta_\theta\p}{\alpha^2(n-1)r^2}\right|^2+ 2\alpha^2\frac{1}{r^2}\left| \nabla_\theta \p' -\frac{\nabla_\theta \p }{r}\right|^2+\frac{1}{r^4}\k_{\mathbb{S}^{d-1}}[\p]\,,
\end{align}
where
$$
\k_{\mathbb{S}^{d-1}}[\p]:=\frac{1}{2}\Delta_\theta|\nabla_\theta\p|^2-\nabla_\theta\p\cdot\nabla_\theta\Delta_\theta\p-\frac{1}{n-1}\left(\Delta_\theta\p\right)^2-(n-2)\alpha^2|\nabla_\theta \p|^2\,. 
$$
Moreover, from \cite[Lemma 5.2]{DEL}, one has
\begin{align}\label{eq23}
\int_{\mathbb{S}^{d-1}}\p^{1-n}\k_{\mathbb{S}^{d-1}}[\p]\,d\theta \geq (n-2)\left(\dfrac{d-1}{n-1}-\alpha^2\right)\int_{\mathbb{S}^{d-1}}\p^{1-n} \frac{1}{r^4} |\nabla_\theta\p |^2\eta^s\, d\theta\,.
\end{align}
Let $\eta=\eta(r)\in C^{\infty}_c(0,\infty)$ be positive and let $s\geq 0$. Multiplying \eqref{eq22} by $\eta^s$ and integrating on $(0,\infty)$, by using \eqref{eq23} we get \eqref{0.0}. In particular \eqref{geq0} immediately follows if \eqref{FS_alpha} holds. The rigidity in \eqref{0.00} is a consequence of \cite[Corollary 5.5]{DEL}.

The case $d=2$ can be proved with a similar argument, using \cite[Lemma 5.3]{DEL} instead of \cite[Lemma 5.1]{DEL}.
\end{proof}

The following key integral estimate will be used in the proofs of Theorems \ref{teo1}, \ref{teo3} and \ref{teo2}.

\begin{lemma}\label{int_ineq}  Assume $n>d\geq 2$ and that \eqref{FS_alpha} holds. Let $\p\in C^3\left((0,\infty)\times\mathbb{S}^{d-1}\right)$ then
$$
0\leq \int_{(0,+\infty)\times\mathbb{S}^{d-1}}\p^{1-n}\k[\p]\eta^s\, d\mu\leq C\int_{(0,+\infty)\times\mathbb{S}^{d-1}}\p^{1-n}|\D\p|^2|\eta'|^2\, d\mu\, ,
$$
for every $s\geq 2$ and $\eta=\eta(r)\in C^{\infty}_c(0,+\infty)$ positive, for some $C>0$.
\end{lemma}

\begin{proof}
The first inequality is the one in \eqref{geq0}. Regarding the second inequality, from Lemma \ref{div_id} we have
$$
\int_{(0,+\infty)\times\mathbb{S}^{d-1}}\p^{1-n}\k[\p]\eta^s\, d\mu=\int_{(0,\infty)\times\mathbb{S}^{d-1}}\D_i\left( \frac{1}{2}\p^{1-n}\D_i|\D\p|^2-\frac{1}{n}\p^{1-n}\L\p\D_i\p\right)\eta^s\, d\mu\, . 
$$
Since $\eta=\eta(r)$, integrating by parts we obtain
\begin{equation}\label{0.2}
\int_{(0,+\infty)\times\mathbb{S}^{d-1}}\p^{1-n}\k[\p]\eta^s\, d\mu=-s\int_{(0,+\infty)\times\mathbb{S}^{d-1}}\p^{1-n}\left[ \frac{1}{2}(|\D\p|^2)'-\frac{1}{n}\L\p\, \p'\right]\eta'\eta^{s-1}\, d\mu\, . 
\end{equation}
From \eqref{def_L} we have 
\begin{equation}\label{0.1}
 \frac{1}{2}(|\D\p|^2)'-\frac{1}{n}\L\p\, \p'= \frac{1}{2}(|\D\p|^2)'- \frac{\alpha^2}{n}  \p''\p' -\frac{\alpha^2(n-1)}{n}\frac{(\p')^2}{r}-\frac{1}{n}\frac{\Delta_\theta \p}{r^2}\p'\, .
\end{equation}
Moreover, since
$$
|\mathrm{D}\p|^2=\alpha^2|\p'|^2+\dfrac{|\nabla_\theta\p|^2}{r^2}\,  ,
$$
then  
$$
\frac{1}{2}(|\D\p|^2)'=\frac{1}{2}\left(\alpha^2|\p'|^2+\dfrac{|\nabla_\theta\p|^2}{r^2}\right)'=\alpha^2\p''\p'+\frac{1}{r^2} \nabla_\theta\p'\cdot\nabla_\theta\p - \frac{1}{r^3}|\nabla_\theta \p|^2\, .
$$
Therefore \eqref{0.1} reads as 
\begin{align*}
\frac{1}{2}(|\D\p|^2)'-\frac{1}{n}\L\p\, \p'=&\dfrac{\alpha^2(n-1)}{n}\left( \p''-\frac{\p'}{r} -\frac{\Delta_\theta\p}{\alpha^2(n-1)r^2}\right)\p'\\
&+ \frac{1}{r}\left( \nabla_\theta\p'-\frac{\nabla_\theta\p}{r}\right)\cdot\frac{\nabla_\theta\p}{r}\, . 
\end{align*}
In addition, from Young and Cauchy-Schwarz inequalities we have 
\begin{align*}
\left|\frac{1}{2}(|\D\p|^2)'-\frac{1}{n}\L\p\, \p'\right||\eta'|\eta^{s-1}\leq& \varepsilon \alpha^4\left| \p''-\frac{\p'}{r} -\frac{\Delta_\theta\p}{\alpha^2(n-1)r^2}\right|^2\eta^s +\varepsilon\alpha^2 \frac{1}{r^2}\left| \nabla_\theta\p'-\frac{\nabla_\theta\p}{r}\right|^2\eta^s \\
& + C_{\epsilon,\alpha} \left(\alpha^2 (\p')^2 + \frac{|\nabla_\theta\p|^2}{r^2}\right)|\eta'|^2\eta^{s-2} \\ 
=&\alpha^4\left| \p''-\frac{\p'}{r} -\frac{\Delta_\theta\p}{\alpha^2(n-1)r^2}\right|^2\eta^s +\varepsilon\alpha^2 \frac{1}{r^2}\left| \nabla_\theta\p'-\frac{\nabla_\theta\p}{r}\right|^2\eta^s\\ 
& + C_{\epsilon,\alpha} |\D\p|^2|\eta'|^2\eta^{s-2} \, , 
\end{align*}
for all $\varepsilon>0$ and for some $C_{\epsilon,\alpha}>0$.

Summing up, from \eqref{0.0} and \eqref{0.2}, choosing $\varepsilon=\varepsilon(n,s)$ small enough we have 
$$
\int_{(0,+\infty)\times\mathbb{S}^{d-1}}\p^{1-n}\k[\p]\eta^s\, d\mu\leq C\int_{(0,+\infty)\times\mathbb{S}^{d-1}}\p^{1-n}|\D\p|^2|\eta'|^2\, d\mu\, ,
$$
for every $s\geq 2$ and for some $C>0$.

\end{proof}

\subsection{A pointwise lower bound for $\L$-superharmonic function}

The following general pointwise lower bound will be used in the proofs of Theorems \ref{teo3} and \ref{teo2}. This is a classical result, for the sake of completeness we include its proof.

\begin{lemma} \label{lb}
With the notation introduced in Subsection \ref{Not} let $f=f(r,\theta)\in C^2((0,+\infty)\times\mathbb{S}^{d-1})$ be a positive $\L$-superharmonic function, i.e. 
$$
\begin{cases}
\L f \leq 0\quad &\text{ in } (0,+\infty)\times\mathbb{S}^{d-1}\\
f>0\, ,
\end{cases}
$$
where $\mathcal{L}$ is given by \eqref{def_L}. Then there exist positive constants $\rho,A>0$ such that
$$
f(r,\theta) \geq \frac{A}{r^{n-2}}\quad\text{on } (\rho,+\infty)\times\mathbb{S}^{d-1}\,.
$$
\end{lemma}
\begin{proof} Let $\rho>0$ be such that $\L f\leq0$ in $(\rho,+\infty)\times\mathbb{S}^{d-1}$ and define  
$$
v(r,\theta):=f(r,\theta)- \frac{A}{r^{n-2}}\,, \quad \text{ for } (r,\theta)\in(\rho,+\infty)\times\mathbb{S}^{d-1}
$$
where $A:=\rho^{n-2}\min_{(\rho,+\infty)\times\mathbb{S}^{d-1}} f>0$. Then $v\geq 0$ in $(\rho,+\infty)\times\mathbb{S}^{d-1}$ and
$$
\L v\leq 0 \quad \text{ in } (\rho,+\infty)\times\mathbb{S}^{d-1}
$$
since 
$$
\L r^{2-n}=\alpha^2(2-n)(1-n)r^{-n}+\alpha^2(n-1)(2-n)r^{-n} =0\,.
$$ 
In addition, from the fact that $\liminf_{r\to\infty} v (r,\theta) \geq 0$, if $\inf_{(\rho,+\infty)\times\mathbb{S}^{d-1}} v<0$, then $v$ attains its negative absolute minimum at a point in $(\rho,\infty)\times\mathbb{S}^{d-1}$. By the strong maximum principle, then $v$ must be constant and negative on its domain, a contradiction. Thus $v\geq 0$ in $(\rho,+\infty)\times\mathbb{S}^{d-1}$ and the conclusion follows.
\end{proof}

\subsection{A weak energy estimate} The following weak energy estimate will be used in the proof of Theorem \ref{teo3}.

\begin{lemma}\label{estgen}
Let $u$ be a positive solution of \eqref{CKN-eq} and let $w$, given by \eqref{def_w}, a solution of \eqref{eq_w}. Then, for every $t<-1$, there exists $C>0$ such that
$$
\int_{(0,R)\times\mathbb{S}^{d-1}}w^{p+t}\,d\mu+\int_{(0,R)\times\mathbb{S}^{d-1}}w^t|\D w|^2\,d\mu \leq C R^{\beta} 
$$
where
$$
\beta:=
\begin{cases}
-\frac{(n-2)}{2}t&\quad\text{if }\,t> -2\\
-(n-2)(1+t)&\quad\text{if }\,t\leq-2\,.
\end{cases}
$$
\end{lemma}
\begin{proof} Let $t<-1$, $s>\max\{2,\frac{2(p+t)}{p-2}\}$, $R>0$ and choose a standard smooth cutoff function $\eta=\eta(r)$ with 
$$
\eta\equiv 1 \, \text{ in $[0,R]$} \,, \quad  \eta\equiv 0 \, \text{ in $[2R,+\infty)$}\, , \quad \text{ and } \quad |\eta'|\leq \frac{C}{R} \, \text{ in $[R,2R]$}\, 
$$
for some $C>0$. Multiplying equation \eqref{eq_w} by $w^{1+t}\eta^s$ and integrating by parts we obtain
\begin{align*}
\int_{(0,+\infty)\times\mathbb{S}^{d-1}}w^{p+t}\eta^s\,d\mu =&-\int_{(0,+\infty)\times\mathbb{S}^{d-1}}w^{1+t}\mathcal{L}w\,\eta^s\,d\mu\\
=&(1+t)\int_{(0,+\infty)\times\mathbb{S}^{d-1}}w^t|\D w|^2\eta^s\,d\mu+s\int_{(0,+\infty)\times\mathbb{S}^{d-1}}w^{1+t}\langle \D w,\D\eta\rangle \eta^{s-1}\,d\mu\,.
\end{align*}
By Cauchy-Schwarz and Young inequalities, we get
$$
w^{1+t}\langle \D w,\D\eta\rangle \eta^{s-1}\leq \eps w^t |\D w|^2\eta^s+C_\eps w^{2+t}|\D\eta|^2\eta^{s-2}\,, 
$$
for every $\varepsilon>0$ and for some $C_\varepsilon>0$. Since $t<-1$ and $s\geq 2$, choosing $\eps>0$ small enough, we obtain
\begin{equation}\label{est_pre_0}
\int_{(0,+\infty)\times\mathbb{S}^{d-1}}w^{p+t}\eta^s\,d\mu+\int_{(0,+\infty)\times\mathbb{S}^{d-1}}w^t|\D w|^2\eta^s\,d\mu \leq C \int_{(0,+\infty)\times\mathbb{S}^{d-1}}w^{2+t}|\D\eta|^2\eta^{s-2}\,d\mu\,.
\end{equation}
We have to consider the two cases: $-2<t<-1$ and $t\leq -2$.

If $-2<t<-1$, then $\frac{p+t}{2+t}>1$ and, by the generalized Young inequality,
%\footnote{Given $a,b\geq 0$ then
%$$
%ab\leq \varepsilon a^p + C_\varepsilon b^q\,  ,
%$$
%for all $\varepsilon>0$ and for some $C_\varepsilon>0$, where $p,q>1$ are such that $\frac{1}{p}+\frac{1}{q}=1.$} 
we have
$$
w^{2+t}|\D\eta|^2\eta^{s-2} \leq \eps w^{p+t}\eta^s+C_\eps |\D\eta|^{\frac{2(p+t)}{p-2}}\eta^{s-\frac{2(p+t)}{p-2}}\,,
$$
for every $\varepsilon>0$ and for some $C_\varepsilon>0$. Therefore, since $s\geq \frac{2(p+t)}{p-2}$, choosing $\eps>0$ small enough, we get
\begin{align*}
\int_{(0,R)\times\mathbb{S}^{d-1}}w^{p+t}\,d\mu+\int_{(0,R)\times\mathbb{S}^{d-1}}w^t|\D w|^2\,d\mu \leq& C \int_{(R,2R)\times\mathbb{S}^{d-1}}|\D\eta|^{\frac{2(p+t)}{p-2}}\,d\mu\\
\leq& C R^{n-\frac{2(p+t)}{p-2}}\\
=&CR^{-\frac{n-2}{2}t}\,,
\end{align*}
where we used \eqref{p_v2}.

If $t\leq -2$, we use the lower bound in Lemma \ref{lb} with $f=w$, in \eqref{est_pre_0} to obtain 
\begin{align*}
\int_{(0,R)\times\mathbb{S}^{d-1}}w^{p+t}\,d\mu+\int_{(0,R)\times\mathbb{S}^{d-1}}w^t|\D w|^2\,d\mu \leq& C R^{-(n-2)(2+t)}\int_{(R,2R)\times\mathbb{S}^{d-1}}|\D\eta|^2\,d\mu\\
\leq& C R^{-(n-2)(2+t)-2+n}\\ 
\leq& C R^{-(n-2)(1+t)}\, ,
\end{align*}
and in both cases the claim follows.

\end{proof}

\

\section{Proofs of the results}

Before presenting the proofs of our main results, we observe that any (positive) solution $u$ of \eqref{CKN-eq} satisfies 
$$
u\in C^{\infty}(\mathbb{R}^d\setminus\{0\})\cap L^{\infty}_{\mathrm{loc}}(\mathbb{R}^d)\, , 
$$
this implies that both $w$, given by \eqref{def_w}, and $\p$, given by \eqref{def_p}, satisfy 
$$
w,\p\in C^{\infty}((0,+\infty)\times\mathbb{S}^{d-1})\cap L^{\infty}_{\mathrm{loc}}([0,+\infty)\times\mathbb{S}^{d-1})\, . 
$$
Moreover, if $u\in\mathcal{D}^{a,b}(\mathbb{R}^d)$ then 
\begin{equation}\label{en_w}
\left(\int_{(0,+\infty)\times\mathbb{S}^{d-1}}|w|^p\, d\mu\right)^{\frac{1}{p}}<+\infty\quad \text{ and } \quad \left(\int_{(0,+\infty)\times\mathbb{S}^{d-1}} |\D w|^2\, d\mu\right)^{\frac{1}{2}}<+\infty\, . 
\end{equation}

\subsection{Proof of Theorem \ref{teo1}}

Assume that $n<4$, then from Lemma \ref{int_ineq} we have 
\begin{equation}\label{est-1}
0\leq \int_{(0,+\infty)\times\mathbb{S}^{d-1}}\p^{1-n}\k[\p]\eta^s\, d\mu\leq C\int_{(0,+\infty)\times\mathbb{S}^{d-1}}\p^{1-n}|\D\p|^2|\eta'|^2\, d\mu\, 
\end{equation}
for every $s\geq 2$, $\eta=\eta(r)\in C^{\infty}_c(0,+\infty)$, for some $C>0$. Let $R>0$, we choose a standard smooth cutoff function $\eta=\eta(r)$ with $$
\eta\equiv 1 \, \text{ in $[0,R]$} \,, \quad  \eta\equiv 0 \, \text{ in $[2R,+\infty)$}\, , \quad \text{ and } \quad |\eta'|\leq \frac{C}{R} \, \text{ in $[R,2R]$}\, 
$$
for some $C>0$. To control the right-hand side of \eqref{est-1} we integrate by parts to obtain
\begin{align*}
\int_{(0,+\infty)\times\mathbb{S}^{d-1}}\p^{1-n}|\D\p|^2\eta^s\, d\mu=&\frac{1}{2-n}\int_{(0,+\infty)\times\mathbb{S}^{d-1}} \langle\D(\p^{2-n}),\D\p\rangle\eta^s\,d\mu \\
=& \frac{1}{n-2}\int_{(0,+\infty)\times\mathbb{S}^{d-1}}\p^{2-n} \mathcal{L}\p\,\eta^s\,d\mu\\
&+ \frac{s}{n-2}\int_{(0,+\infty)\times\mathbb{S}^{d-1}}\p^{2-n}\langle\D\p,\D\eta\rangle\eta^{s-1}\,d\mu \\
=& \frac{2(n-1)^2}{(n-2)^2}\int_{(0,+\infty)\times\mathbb{S}^{d-1}} \p^{1-n}\eta^s\d\mu \\
&+\frac{n}{2(n-2)}\int_{(0,+\infty)\times\mathbb{S}^{d-1}} \p^{1-n}|\D\p|^2\eta^s \\
&+ \frac{s}{n-2}\int_{(0,+\infty)\times\mathbb{S}^{d-1}}\p^{2-n}\langle\D\p,\D\eta\rangle\eta^{s-1}\,d\mu \, , 
\end{align*}
where we used \eqref{eq_p}. Hence 
\begin{align*}
\frac{4-n}{2(n-2)}\int_{(0,+\infty)\times\mathbb{S}^{d-1}}\p^{1-n}|\D\p|^2\eta^s\, d\mu=& - \frac{2(n-1)^2}{(n-2)^2}\int_{(0,+\infty)\times\mathbb{S}^{d-1}} \p^{1-n}\eta^s\d\mu\nonumber  \\
&- \frac{s}{n-2}\int_{(0,+\infty)\times\mathbb{S}^{d-1}}\p^{2-n}\langle\D\p,\D\eta\rangle\eta^{s-1}\,d\mu\, . 
\end{align*}
From Cauchy-Schwarz and Young inequalities we get 
$$
-\frac{s}{n-2}\p^{2-n}\langle\D\p,\D\eta\rangle\eta^{s-1}\leq \p^{2-n}|\D\p||\D\eta|\eta^{s-1} \leq \varepsilon\p^{1-n}|\D\p|^2\eta^s + C_\varepsilon \p^{3-n}|\D\eta|^2\eta^{s-2}\, , 
$$
for every $\varepsilon>0$ and for some $C_\varepsilon>0$. Therefore, since $n<4$, we can choose $\eps>0$ small enough to obtain
$$
\int_{(0,+\infty)\times\mathbb{S}^{d-1}}\p^{1-n}|\D\p|^2\eta^s\, d\mu\leq - C_1\int_{(0,+\infty)\times\mathbb{S}^{d-1}} \p^{1-n}\eta^s\d\mu+C_2\int_{(0,+\infty)\times\mathbb{S}^{d-1}}\p^{3-n}|\D\eta|^2\eta^{s-2}\,d\mu\, . 
$$
for some $C_1,C_2>0$. Let $s>n-1$. By the generalized Young inequality, we have
$$
\p^{3-n}|\D\eta|^2\eta^{s-2} \leq \eps \p^{1-n}\eta^{s}+C_\eps |\D\eta|^{n-1}\eta^{s-n+1}\,,
$$
for every $\varepsilon>0$ and for some $C_\varepsilon>0$. Choosing again $\eps>0$ small enough, we get
\begin{align*}
\int_{(0,+\infty)\times\mathbb{S}^{d-1}}\p^{1-n}|\D\p|^2\eta^s\, d\mu\leq& C\int_{(0,+\infty)\times\mathbb{S}^{d-1}}|\D\eta|^{n-1}\eta^{s-n+1}\,d\mu\\
\leq& C R^{1-n}\int_{R}^{2R}\int_{\mathbb{S}^{d-1}} r^{n-1}\,dr\,d\theta \\
\leq& C R\,.
\end{align*}
In particular, for every $R>0$
$$
\int_{(0,R)\times\mathbb{S}^{d-1}}\p^{1-n}|\D\p|^2\, d\mu \leq C R
$$
for some $C>0$. From \eqref{est-1}, we deduce
\begin{align*}
0\leq \int_{(0,2R)\times\mathbb{S}^{d-1}}\p^{1-n}\k[\p]\, d\mu=&\int_{(0,+\infty)\times\mathbb{S}^{d-1}}\p^{1-n}\k[\p]\eta^s\, d\mu\\
\leq&CR^{-2}\int_{(0,2R)\times\mathbb{S}^{d-1}}\p^{1-n}|\D\p|^2\, d\mu \\
\leq& C R^{-1}\,.
\end{align*}
Taking the limit as $R\to\infty$ we get
$$
\int_{(0,+\infty)\times\mathbb{S}^{d-1}}\p^{1-n}\k[\p]\, d\mu =0.
$$
From Proposition \ref{prop-del} we obtain
$$
u(x)=\mathcal{U}(x)\, ,
$$
up to scaling, where $\mathcal{U}$ is given by \eqref{Tal_CKN} and the conclusion of Theorem \ref{teo1} follows.

\

\subsection{Proof of Theorem \ref{teo3}}

Let $n\geq 4$ and assume
$$
u(x)\leq C|x|^{\sigma\alpha}\quad\text{for some }\sigma<-\frac{(n-2)(n-6)}{2(n-4)}\,,
$$
outside a compact set. Equivalently, for $r\geq\rho>0$, 
\begin{equation}\label{ubw}
w(r,\theta)\leq Cr^{\sigma}\quad\text{for some }\sigma<-\frac{(n-2)(n-6)}{2(n-4)}\,.
\end{equation}
From Lemma \ref{int_ineq} we have 
\begin{equation}\label{est-3}
0\leq \int_{(0,+\infty)\times\mathbb{S}^{d-1}}\p^{1-n}\k[\p]\eta^s\, d\mu\leq C\int_{(0,+\infty)\times\mathbb{S}^{d-1}}\p^{1-n}|\D\p|^2|\eta'|^2\, d\mu\, 
\end{equation}
for every $s\geq 2$, $\eta=\eta(r)\in C^{\infty}_c(0,+\infty)$, for some $C>0$. Let $R>0$, we choose again a standard smooth cutoff function $\eta=\eta(r)$ with 
$$
\eta\equiv 1 \, \text{ in $[0,R]$} \,, \quad  \eta\equiv 0 \, \text{ in $[2R,+\infty)$}\, , \quad \text{ and } \quad |\eta'|\leq \frac{C}{R} \, \text{ in $[R,2R]$}\, 
$$
for some $C>0$. Therefore
$$
0\leq\int_{(0,2R)\times\mathbb{S}^{d-1}}\p^{1-n}\k[\p]\, d\mu\leq C R^{-2}\int_{(R,2R)\times\mathbb{S}^{d-1}}w^{-\frac{2}{n-2}}|\D w|^2\, d\mu
$$
Let $\gamma> \frac{n-4}{n-2}$. To control the right-hand side we use the integral estimates in Lemma \ref{estgen} with 
$$
t=-\frac{2}{n-2}-\gamma<-1
$$
which yields
\begin{align}\label{ee}\nonumber
R^{-2}\int_{(R,2R)\times\mathbb{S}^{d-1}}w^{-\frac{2}{n-2}}|\D w|^2\, d\mu\leq&R^{-2} \sup_{(R,2R)\times\mathbb{S}^{d-1}} w^\gamma \int_{(R,2R)\times\mathbb{S}^{d-1}}w^{-\frac{2}{n-2}-\gamma}|\D w|^2\, d\mu\\
\leq& CR^{\beta-2}\sup_{(R,2R)\times\mathbb{S}^{d-1}} w^\gamma\,.
\end{align}
If $\frac{n-4}{n-2}<\gamma< \frac{2(n-3)}{n-2}$, i.e. $t>-2$, then the exponent $\beta$ given in Lemma \ref{estgen} is $\beta=-\frac{n-2}{2}t=1+\frac{n-2}{2}\gamma$ and
$$
R^{-2}\int_{(R,2R)\times\mathbb{S}^{d-1}}w^{-\frac{2}{n-2}}|\D w|^2\, d\mu\leq C R^{-1+\frac{n-2}{2}\gamma+\sigma\gamma}\,.
$$
Choosing $\frac{n-4}{n-2}<\gamma< \frac{2(n-3)}{n-2}$ sufficiently close to $\frac{n-4}{n-2}$ we have that the exponent of $R$ is negative, so the left-hand side tends to zero as $R\to\infty$. Using \eqref{est-3}, we get 
$$
\int_{(0,+\infty)\times\mathbb{S}^{d-1}}\p^{1-n}\k[\p]\, d\mu =0.
$$
From Proposition \ref{prop-del} we obtain
$$
u(x)=\mathcal{U}(x)\, ,
$$
up to scaling, where $\mathcal{U}$ is given by \eqref{Tal_CKN} and the conclusion of Theorem \ref{teo3} follows.

\begin{remark} It is easy to see that, choosing $\gamma\geq \frac{2(n-3)}{n-2}$ in \eqref{ee}, one can obtain similar but worst estimates.
\end{remark}

\

\subsection{Proof of Corollaries \ref{cor1} and \ref{cor2}} If $n<4$ the results follows immediately from Theorem \ref{teo1}. If $n\geq 4$, we apply Theorem \ref{teo3} with 
$$
\sigma\alpha:=-\frac{d-2-2a}{2}<-\frac{(n-2)(n-6)}{2(n-4)}\alpha=-\frac{d-2-2a}{2}\left(1-\frac{2}{n-4}\right)
$$
or 
$$
\sigma\alpha:=0,
$$
to prove Corollaries \ref{cor1} and \ref{cor2}, respectively.

\

\subsection{Proof of Theorem \ref{teo2}}

From Lemma \ref{int_ineq} we have 
\begin{equation}\label{est-2}
0\leq \int_{(0,+\infty)\times\mathbb{S}^{d-1}}\p^{1-n}\k[\p]\eta^s\, d\mu\leq C\int_{(0,+\infty)\times\mathbb{S}^{d-1}}\p^{1-n}|\D\p|^2|\eta'|^2\, d\mu\, 
\end{equation}
for every $s\geq 2$, $\eta=\eta(r)\in C^{\infty}_c(0,+\infty)$, for some $C>0$. Let $R>0$, we choose again a standard smooth cutoff function $\eta=\eta(r)$ with 
$$
\eta\equiv 1 \, \text{ in $[0,R]$} \,, \quad  \eta\equiv 0 \, \text{ in $[2R,+\infty)$}\, , \quad \text{ and } \quad |\eta'|\leq \frac{C}{R} \, \text{ in $[R,2R]$}\, 
$$
for some $C>0$. To control the right-hand side of \eqref{est-2} we use the lower bound in Lemma \ref{lb} with $f=w$ which yields
$$
w(r,\theta)\geq \frac{A}{r^{n-2}}\quad\text{for all } (r,\theta)\in(\rho,+\infty)\times\mathbb{S}^{d-1}\, ,
$$ 
for some $\rho>0$ and $A>0$. Thus, from the definition of $\p$ in \eqref{def_p}, we deduce 
\begin{align*}
\int_{(0,+\infty)\times\mathbb{S}^{d-1}}\p^{1-n}|\D\p|^2|\eta'|^2\, d\mu\leq&CR^{-2}\int_{(R,2R)\times\mathbb{S}^{d-1}}\p^{1-n}|\D\p|^2\, d\mu\\
=&C' R^{-2}\int_{(R,2R)\times\mathbb{S}^{d-1}}w^{-\frac{2}{n-2}}|\D w|^2\, d\mu\\
\leq& C\int_{(R,2R)\times\mathbb{S}^{d-1}}|\D w|^2\, d\mu\,.
\end{align*}
Since $u\in \mathcal{D}^{a,b}(\RR^d)$, then from \eqref{en_w} we have 
$$
\int_{(0,+\infty)\times\mathbb{S}^{d-1}}|\D w|^2\, d\mu<+\infty\,.
$$
Therefore, taking the limit $R\to\infty$ and using \eqref{est-2}, we get 
$$
\int_{(0,+\infty)\times\mathbb{S}^{d-1}}\p^{1-n}\k[\p]\, d\mu =0.
$$
From Proposition \ref{prop-del} we obtain
$$
u(x)=\mathcal{U}(x)\, ,
$$
up to scaling, where $\mathcal{U}$ is given by \eqref{Tal_CKN} and the conclusion of Theorem \ref{teo2} follows.

\

\

\begin{ackn}
\noindent The first and third authors are member of the {\em GNSAGA, Gruppo Nazionale per le Strutture Algebriche, Geometriche e le loro Applicazioni} of INdAM. The second author is member of {\em GNAMPA, Gruppo Nazionale per l'Analisi Matematica, la Probabilit\`a e le loro Applicazioni} of INdAM. Moreover, the first and the third authors are partially supported by the project PRIN 2022 ``Differential-geometric aspects of manifolds via Global Analysis'', the second author is partially supported by the project PRIN 2022 ``Geometric-Analytic Methods for PDEs and Applications''.
\end{ackn}

\

\

\noindent{\bf Data availability statement}

\noindent Data sharing not applicable to this article as no datasets were generated or analysed during the current study.

\

\

\

\


\begin{thebibliography}{20}

 \bibitem{aubin} T. Aubin, \textit{Probl\`emes isop\'erim\'etriques et espaces de Sobolev}, J. Diff. Geom. \textbf{11} (1976), 573--598.
 
 
\bibitem{CGS} L. Caffarelli, B. Gidas, J. Spruck, \emph{Asymptotic symmetry and local behavior of semilinear elliptic equations with critical Sobolev growth}, Comm. Pure Appl. Math. \textbf{42} (1989), 271--297.

\bibitem{CKN} L. Caffarelli, R. Kohn, L. Nirenberg. \emph{First order interpolation inequalities with weights,} Compos. Math. \textbf{53} 259--275, (1984).


\bibitem{CLMR} G. Catino, Y. Y. Li, D. D. Monticelli, A. Roncoroni. \emph{A Liouville theorem in the
Heisenberg group.} J. Eur. Math. Soc. To appear.


\bibitem{CaMo} G. Catino, D.D. Monticelli. \emph{Semilinear elliptic equations on manifolds with nonnegative Ricci curvature.} J. Eur. Math. Soc., to appear.

\bibitem{CMR} G. Catino, D.D. Monticelli, A. Roncoroni. \emph{On the critical $p$-Laplace equation.} Adv. Math. \textbf{433} (2023), 109331.


\bibitem{CMWR} G. Catino, D.D. Monticelli, A. Roncoroni, X. Wang. \emph{Liouville theorems on pseudohermitian manifolds with nonnegative Tanaka-Webster curvature.} Preprint.

\bibitem{CW} F. Catrina, Z.-Q. Wang. \emph{On the Caffarelli-Kohn-Nirenberg inequalities: sharp constants, existence (and nonexistence), and symmetry of extremal functions.} Comm. Pure Appl.
Math. \textbf{54} (2), 229--258 (2001).

\bibitem{CL} W. X. Chen, C. Li. \emph{Classification of solutions of some nonlinear elliptic equations.} Duke Math. J. \textbf{63} (1991), no. 3, 615--622.

\bibitem{CC} K.S. Chou, C.W.Chu. \emph{On the best constant for a weighted Sobolev-Hardy inequality.} J. Lond. Math. Soc. \textbf{48} (1), 137--151 (1993).

\bibitem{CirCor} G. Ciraolo, R. Corso. \emph{Symmetry for positive critical points of Caffarelli-Kohn-Nirenberg inequalities,} Nonlinear Anal. \textbf{216} (2022), 112683.

\bibitem{CCR} G. Ciraolo, R. Corso, A. Roncoroni. \emph{Non-existence of patterns for a class of weighted degenerate operators.} J. Differ. Equations \textbf{370} (2023), 240--270

\bibitem{CFP} G. Ciraolo, A. Farina , C.C. Polvara. \emph{Classification results, rigidity theorems and semilinear PDEs on Riemannian manifolds: a $P-$function approach.} Preprint


\bibitem{CFR} G. Ciraolo, A. Figalli, A.Roncoroni. \emph{Symmetry results for critical anisotropic $p$-Laplacian equations in convex cones,} Geom. Funct. Anal. \textbf{30} (2020), 770--803.


\bibitem{CPP} G. Ciraolo, F. Pacella, C.C. Polvara. \emph{Symmetry breaking and instability for semilinear elliptic equations in spherical sectors and cones.} J. Math. Pures Appl. \textbf{187} (2024), 138--170. 

\bibitem{CirPol} G. Ciraolo, C.C. Polvara. \emph{On the classification of extremals of Caffarelli-Kohn-Nirenberg inequalities.} Preprint.

\bibitem{DMMS}  L. Damascelli, S. Merch\'an, L. Montoro, B. Sciunzi. \emph{Radial symmetry and applications for a problem involving the $-\Delta_p(\cdot)$-operator and critical nonlinearity in $\mathbb{R}^n$.} Adv. Math. \textbf{265} (10) (2014), 313--335.


\bibitem{DEL} J. Dolbeault, M.J. Esteban, M. Loss. \emph{Rigidity versus symmetry breaking via nonlinear flows on cylinders and Euclidean spaces,} Invent. math. \textbf{206} 397--440, (2016).

\bibitem{FS} V. Felli, M. Schneider. \emph{Perturbation results of critical elliptic equations of Caffarelli-Kohn-Nirenberg type,} J. Differ. Equations \textbf{191} (1) 121--142, (2003).

\bibitem{FV} J. Flynn, J. V\'etois. \emph{Liouville-type results for the CR Yamabe equation in the Heisenberg group.} Ann. Sc. Norm. Super. Pisa Cl. Sci. To appear.


\bibitem{GNN} B. Gidas, W.M. Ni, L. Nirenberg, \emph{Symmetry of positive solutions of nonlinear elliptic equations.} Adv. Math. Suppl. Stud. 7A (1981), 369--402.

\bibitem{JL} D. Jerison, J.M. Lee. \emph{Extremals for the Sobolev inequality on the Heisenberg group
and the CR Yamabe problem.} J Amer. Math. Soc. \textbf{1} (1988), 1--13.

\bibitem{LZ} Y.Y. Li, L. Zhang, \textit{Liouville-type theorems and Harnack inequalities for semilinear elliptic equations}, J. Anal. Math. \textbf{90} (2003), 27--87.

\bibitem{LPT}  P.L. Lions, F. Pacella, M. Tricarico. \emph{Best constants in Sobolev inequalities for functions vanishing on some part of the boundary and related questions.} Indiana Univ. Math. J. \textbf{37} (2) (1988), 301--324.

\bibitem{obata} M. Obata. \emph{The conjectures on conformal transformations of Riemannian manifolds.}
J. Diff. Geom. \textbf{6} (1971), 247--258.


\bibitem{Ou} Q. Ou. \emph{On the classification of entire solutions to the critical $p-$Laplace equation.} Math. Ann. \textbf{392} (2025), 1711--1729.

\bibitem{Praj} J. V. Prajapat, A. S. Varghese. \emph{Symmetry and classification of solutions to an integral equation in the Heisenberg group $\mathbb{H}^n$.} Preprint.

\bibitem{rodemich} E. Rodemich. \emph{The Sobolev inequalities with best possible constants.} Analysis Seminar at California Institute of Technology (1966).

\bibitem{Sci} B. Sciunzi. \emph{Classification of positive $\mathcal{D}^{1,p(\mathbb{R}^n)}$-solutions to the critical $p$-Laplace equation in $\mathbb{R}^n$.} Adv. Math. \textbf{291} (2016), 12--23.



\bibitem{SV} S. Shakerian, J. V\'etois. \emph{Sharp pointwise estimates for weighted critical p-Laplace equations.} Nonlinear Anal. \textbf{206} (2021), 112236.

\bibitem{cinesi} L. Sun, Y. Wang. \emph{Critical quasilinear equations on Riemannian manifolds.} Preprint.

\bibitem{talenti} G. Talenti, \emph{Best constant in Sobolev inequality}, Ann. Mat. Pura Appl. \textbf{110} (1976), 353--372.

\bibitem{Vet_bis} J. V\'etois. \emph{A note on the classification of positive solutions to the critical $p-$Laplace equation in $\mathbb{R}^n$.} Adv. Nonlinear Stud. (2024), \textbf{24} (3), 543--552.


\bibitem{Vet} J. V\'etois. \emph{A priori estimates and application to the symmetry of solutions for critical $p$-Laplace equations,} J. Differ. Equ. \textbf{260} (2016) 149--161.


\end{thebibliography}
\end{document}